\documentclass[10pt]{amsart}%
\usepackage{amssymb}
\usepackage{amsfonts}
\usepackage{amsmath}
\usepackage[centertags]{amsmath}
\usepackage{newlfont}
\usepackage{url}%
\usepackage{graphicx}
\providecommand{\U}[1]{\protect\rule{.1in}{.1in}}
\theoremstyle{plain}

\newtheorem{corollary}[subsection]{Corollary}

\newtheorem{definition}[subsection]{Definition}

\newtheorem{lemma}[subsection]{Lemma}

\newtheorem{proposition}[subsection]{Proposition}

\numberwithin{equation}{subsection}

\theoremstyle{plain}

\newtheorem{cor}[subsubsection]{Corollary}

\newtheorem{lem}[subsubsection]{Lemma}

\theoremstyle{definition}

\theoremstyle{remark}

\theoremstyle{plain}
\newtheorem{teo}{Theorem}
\begin{document}
\title{\textbf{On Stabilization of }$E_{k}$\textbf{\ Chains}}
\author{Tuba \c{C}AKMAK}
\address{Atat\"{u}rk University, Faculty of Science, Department of Mathematics, 25240
Erzurum, Turkey }
\email{cakmaktuba@yahoo.com}
\thanks{This work was supported by T\"{U}B\.{I}TAK, the Scientific and Technological
Research Council of Turkey, through its programs 2214A and 2211E. }

\begin{abstract}
We study special subgroups of infinite groups that generalize double
centralizers. We analyze sufficient conditions for descending chains of such
subgroups to stop after finitely many steps. We discuss whether this
phenomenon can happen in the class of groups satisfying chain condition on centralizers.

\end{abstract}
\maketitle

\section{Introduction}

In modern infinite group theory, chain conditions have played an important
role. A natural chain condition is the one on centralizers. A group is said to
satisfy the descending chain condition on centralizers if every proper
descending chain of centralizers stabilizes after finitely many steps. Such
groups are denoted $\mathfrak{M}_{c}$-groups. By elementary properties of
centralizers, the descending chain condition on centralizers is equivalent to
the ascending one. Since many natural classes of groups such as linear groups
and finitely generated abelian-by-nilpotent groups enjoy this property, it has
been separately analyzed in such fundamental papers as \cite{Bryant 1} and
\cite{Bryant 3}.

\bigskip

$\mathfrak{M}_{c}$-groups are frequently encountered in model theory since the
class of stable groups, a class of groups of fundamental importance in model
theory, have this property. The model-theoretic analysis of mathematical
structures frequently uses the notion of first-order definability in the sense
of mathematical logic. When these structures are groups, a frequent question
is whether algebraic properties (e.g. nilpotency, solvability) of subgroups
are inherited by sufficiently small definable supergroups containing them
(\emph{envelopes}). In \cite{A-B}, Alt\i nel and Baginski showed that in an
$\mathfrak{M}_{c}$-group every nilpotent subgroup is contained in a definable
nilpotent subgroup of the same nilpotency class. This generalizes a similar
property of the Zariski closure of nilpotent subgroups of algebraic groups
over algebraically closed fields.

\bigskip

In their proof, Alt\i nel and Baginski start with a nilpotent subgroup $H$ of
an $\mathfrak{M}_{C}$-group $G$, and construct a descending chain of
supergroups of $H$ denoted $E_{k}(H)$ ($k\in\mathbb{N}$) reminiscent of double
centralizers. The definability of the envelope constructed depends intimately
on the ambient chain condition on centralizers as well as on the nature of the
subgroups $E_{k}(H)$, while its nilpotency is related to this very nature and
the nilpotency of $H$.

\bigskip

Our initial motivation was to measure to what extent these techniques would
help prove similar definability results for other classes of subgroups (e.g.
solvable subgroups of $\mathfrak{M}_{C}$-groups) as well as for classes of
ambient groups satisfying weaker chain conditions. Along the way, we proved
Theorem 1 (section 3.1) that states that the nilpotency of the envelope is a
consequence of the nilpotency of $H$, and from this it follows through a
simple induction argument that when $H$ is $k$-nilpotent, the descending chain
$(E_{i}(H))_{i}$ stabilizes after at most $k$ steps. These refinements of part
of Alt\i nel and Baginski's proof led us to investigating more this
stabilization property, especially under the assumption that the ambient group
is an $\mathfrak{M}_{C}$-group. Indeed, as proven in section 3.2, in groups
enjoying various topological properties and satisfying Noetherianity
conditions on closed subgroups, such as linear groups, the $E_{k}$-chains of
envelopes of arbitrary subgroups stabilize. We interpret these affirmative
answers to the stabilization problem of the $E_{k}$-chains as potentially
useful for our initial purposes.

\bigskip

Nevertheless, the stabilization of $E_{k}$-chains is a strong property. One
suspects that it is false in general, and understanding the conditions under
which a counterexample would arise is likely to yield further information. In
Section 4, where we construct an example in $Sym(\mathbb{N})$ with an infinite
$E_{k}$-chain of envelopes. Somewhat to our surprise, the counterexample to
the stabilization is rather involved. It should also be noted that
$Sym(\mathbb{N})$ is far from satisfying the descending chain condition on centralizers.

\bigskip

We have not been able to prove the stabilization property for envelopes of
arbitrary subgroups of $\mathfrak{M}_{C}$-groups. Potentially, this could have
led to further results similar to those in \cite{A-B}, and it remains to be
done. In a more positive vein, we are able to extend the results of Section
3.1 to hypercentral subgroups (of $\mathfrak{M}_{C}$-groups). These extensions
necessitate introducing transfinite versions of iterated centralizers as well
as of the $E_{k}$-chains, and are done in \cite{Cak}.

\bigskip

\section{Key Facts}

Our group-theoretic notation is standard. We write $H\leq G$ to denote that
$H$ is a subgroup of $G$ and $H\vartriangleleft G$ to denote $H$ is normal in
$G$. If $H\subseteq G$, then $\left\langle H\right\rangle $ denotes the
subgroup generated by $H$. For any subset $H$ of $G$, the centralizer of $H$
is $C_{G}\left(  H\right)  =\left\{  g\in G\mid\forall h\in H\text{
\ }gh=hg\right\}  $, while the normalizer of $H$ is $N_{G}\left(  H\right)
=\left\{  g\in G\mid\forall h\in H\text{ \ }g^{-1}hg\in H\right\}  $. Given
$g,h\in G$ and the commutator of these elements is $\left[  g,h\right]
:=g^{-1}h^{-1}gh$.

\bigskip

The key notion of this article is a special enveloping supergroup of a
subgroup $H$ of an arbitrary group $G$, that will be denoted $E_{k}(H)$ for
every $k\in\mathbb{N}$. To define it, we need to recall the notion of iterated
centralizer introduced in \cite{Bryant 1}:

%which is given
%in Definition \ref{2.1} and constructed by using the following
%generalization of central series, called iterated centralizers:

\bigskip

\begin{definition}
[Bryant, 1979]Let $G$ be a group and $A$ any subset of $G$. Set $C_{G}%
^{0}\left(  A\right)  =1$ and for $k\geq1,$ the iterated centralizer of $A$ in
$G$ is
\[
C_{G}^{k}\left(  A\right)  =\left\{  x\in\underset{n<k}{\cap}N_{G}\left(
C_{G}^{n}\left(  A\right)  \right)  \mid\left[  x,A\right]  \subseteq
C_{G}^{k-1}\left(  A\right)  \right\}  .
\]

\end{definition}

\bigskip

One can show by induction that the iterated centralizers $C_{G}^{k}\left(
H\right)  $ form an ascending sequence: $1=C_{G}^{0}\left(  H\right)  \leq
C_{G}^{1}\left(  H\right)  \leq...\leq G.$

\bigskip

Some of the basic properties of iterated centralizers are stated in the
following lemma. In particular, one sees that the notion of iterated
centralizer generalizes the more commonly known notion of the $k$th center of
a group.

\bigskip

\begin{lemma}
[Bryant, 1979]\bigskip Let $G$ be a group and $H$ a subgroup and $k\geq0.$ For
the $kth$ iterated centralizer of $H$ in $G,$ the following relations hold:

\begin{enumerate}
\item[(i)] $C_{G}^{k}\left(  H\right)  \leq G$

\item[(ii)] $C_{G}^{k}\left(  H\right)  \cap H=Z_{k}\left(  H\right)  $

\item[(iii)] When $H=G,$ $C_{G}^{k}\left(  G\right)  =Z_{k}\left(  G\right)  $

\item[(iv)] If $H$ is a nilpotent subgroup of class $k,$ then $H\leq C_{G}%
^{k}\left(  H\right)  $.
\end{enumerate}
\end{lemma}

\bigskip

Let us recall that if $G$ is a group then $Z_{0}(G)=\{1\}$ and inductively,
for every $k\in\mathbb{N}$, $Z_{k+1}\left(  G\right)  =\left\{  g\in
G\mid\left[  g,G\right]  \subseteq Z_{k}\left(  G\right)  \right\}  $. The
series $(Z_{k}(G))_{k\in\mathbb{N}}$ is known as the upper central series; a
group is nilpotent if and only if there exists $k\in\mathbb{N}$ such that
$Z_{k}\left(  G\right)  =G$. The least such $k$ is the nilpotency class of $G$.

%When $H=G$, the $k$th iterated centralizer of $G$ is more commonly known as $%
%Z_{k}\left( G\right) $, the $k$th center of $G$. $(k+1)$th center of group $G
%$ defined as for all $k\geq 0$ $Z_{k+1}\left( G\right) =\left\{ g\in G\mid %
%\left[ g,G\right] \subseteq Z_{k}\left( G\right) \right\} $ where $%
%Z_{0}\left( G\right) =\left\{ 1\right\} $. The subgroup $Z_{1}\left(
%G\right) =Z\left( G\right) $ is the center of $G$. This series is known as
%the upper central series; a group is nilpotent of class $k$ if and only if $%
%Z_{k}\left( G\right) =G$.

\bigskip

%The $E_{k}$ definable envelopes were introduced in \cite{A-B} to show that
%nilpotent subgroups of $\mathfrak{M}_{C}$-groups are contained in definable
%(in the sense of model theory) nilpotent subgroups of the same nilpotency
%class.

We now define the central notion and tool of this article:

\bigskip

\begin{definition}
[Alt\i nel-Baginski, 2014, Definition 3.5]\label{2.1}Let $G$ be a group and
$H$ a subgroup. For $k\in%
%TCIMACRO{\U{2115} }%
%BeginExpansion
\mathbb{N}
%EndExpansion
,$ a sequence of subgroups $E_{k}\left(  H\right)  $ of $G$ is defined
\[
E_{k+1}(H)=\left\{  g\in E_{k}(H)\mid\left[  g,C_{E_{k}(H)}^{k+1}\left(
H\right)  \right]  \leq C_{E_{k}(H)}^{k}\left(  H\right)  \right\}
\]
where $E_{0}\left(  H\right)  =G.$
\end{definition}

\bigskip

By definition, these subgroups of $G$ form a descending sequence such as
\[
G=E_{0}\left(  H\right)  \geq E_{1}\left(  H\right)  \geq...\geq H.
\]

It follows easily from this definition that $E_{1}\left(  H\right)
=C_{G}\left(  C_{G}\left(  H\right)  \right)  .$

\bigskip

%The following technical facts which show some of relations between the
%iterated centers and iterated centralizers provide a basis for our results.

Before finishing this section, we will recall some basic facts from \cite{A-B}
and draw some corollaries.

\bigskip

\begin{lemma}
[Alt\i nel-Baginski, 2014, Lemma 2.5]\label{2.4}Let $A\leq B\leq C$ be groups
and suppose that for all $j\leq k$ we have $C_{C}^{j}\left(  A\right)  =$
$Z_{j}\left(  C\right)  .$ Then

\begin{enumerate}
\item[(i)] $C_{C}^{j}\left(  A\right)  =C_{C}^{j}\left(  B\right)
=Z_{j}\left(  C\right)  ,\forall j\leq k$

\item[(ii)] $C_{B}^{j}(A)=Z_{j}\left(  B\right)  =Z_{j}\left(  C\right)  \cap
B,\forall j\leq k$

\item[(iii)] $C_{B}^{k+1}(A)=C_{C}^{k+1}(A)\cap B,\forall j\leq k.$
\end{enumerate}
\end{lemma}

\bigskip

\begin{corollary}
\label{2.5}Let $i,j\in%
%TCIMACRO{\U{2115} }%
%BeginExpansion
\mathbb{N}
%EndExpansion
$ such that $i\leq j.$ Then $Z_{i}\left(  E_{i}\left(  H\right)  \right)  \leq
Z_{j}\left(  E_{j}\left(  H\right)  \right)  .$
\end{corollary}

\begin{proof}
Since the iterated centers form an ascending chain $Z_{i}\left(  E_{j}\left(
H\right)  \right)  \leq Z_{j}\left(  E_{j}\left(  H\right)  \right)  $ and
inductively $Z_{i}\left(  E_{i}\left(  H\right)  \right)  \leq E_{j}\left(
H\right)  $, applying the Lemma \ref{2.4} (ii) to $H\leq E_{j}\left(
H\right)  \leq E_{i}\left(  H\right)  $ subgroups,
\[
Z_{j}\left(  E_{j}\left(  H\right)  \right)  \geq Z_{i}\left(  E_{j}\left(
H\right)  \right)  =Z_{i}\left(  E_{i}\left(  H\right)  \right)  \cap
E_{j}\left(  H\right)  =Z_{i}\left(  E_{i}\left(  H\right)  \right)
\]
is obtained.
\end{proof}

\bigskip

\begin{lemma}
[Alt\i nel-Baginski, 2014]\label{2.6}Let $G$ be an arbitrary group and $H$ a
subgroup of $G.$ Then%
\[
C_{E_{k}(H)}^{j}(H)=Z_{j}\left(  E_{k}(H)\right)
\]
for all $j\leq k$.
\end{lemma}

\bigskip

A practical conclusion of this lemma is the following corollary. Throughout
the paper, it will be used without mention.

\bigskip

\begin{corollary}
Let $G$ be a group and $H\leq G.$ Then the iterated centralizer is
\[
C_{E_{k}}^{i+1}\left(  H\right)  =\left\{  x\in E_{k}\mid\left[  x,H\right]
\subseteq Z_{i}\left(  E_{k}\left(  H\right)  \right)  \right\}
\]
for all $i\leq k.$
\end{corollary}

\begin{proof}
By definition $C_{E_{k}}^{i+1}\left(  H\right)  =\left\{  x\in\underset{j\leq
i}{\cap}N_{E_{k}}\left(  C_{E_{k}}^{j}\left(  H\right)  \right)  \mid\left[
x,H\right]  \subseteq C_{E_{k}}^{i}\left(  H\right)  \right\}  .$ By the
previous lemma $C_{E_{k}(H)}^{i}(H)=Z_{i}\left(  E_{k}(H)\right)  $ for all
$i\leq k$\textbf{\ }
\begin{align*}
C_{E_{k}}^{i+1}\left(  H\right)   & =\left\{  x\in\underset{j\leq i}{\cap
}N_{E_{k}}\left(  C_{E_{k}}^{j}\left(  H\right)  \right)  \mid\left[
x,H\right]  \subseteq C_{E_{k}}^{i}\left(  H\right)  \right\} \\
& =\left\{  x\in\underset{j\leq i}{\cap}N_{E_{k}}\left(  Z_{j}\left(
E_{k}(H)\right)  \right)  \mid\left[  x,H\right]  \subseteq Z_{i}\left(
E_{k}\left(  H\right)  \right)  \right\} \\
& =\left\{  x\in E_{k}\mid\left[  x,H\right]  \subseteq Z_{i}\left(
E_{k}\left(  H\right)  \right)  \right\}
\end{align*}
is obtained.
\end{proof}

\section{Affirmative Answers}

In this section, we show that under some conditions the chains of $E_{k}%
$-envelopes of some specific subgroups stabilize. First, we shall show that
the chain of $E_{k}$-envelopes of any nilpotent subgroup of an arbitrary group
$G$ stabilizes. Then, we will analyze topological conditions that yield the
stabilization property.

\bigskip

\subsection{Nilpotent Subgroups\label{3.1}}

In this subsection, unless otherwise mentioned, $H$ will stand for a nilpotent
subgroup of a fixed arbitrary ambient group $G$. For simplicity, we will
denote the various $E_{k}\left(  H\right)  $ by $E_{k}$.

\bigskip

\begin{lem}
\label{3.1.1}Let $G$ be a group and $H$ an abelian subgroup of $G.$Then
$C_{G}\left(  C_{G}\left(  H\right)  \right)  $ is abelian.
\end{lem}

\begin{proof}
Since $H$ is abelian, $H\leq C_{G}\left(  H\right)  $, and then $C_{G}\left(
H\right)  \geq C_{G}\left(  C_{G}\left(  H\right)  \right)  .$ Thus
\[
Z\left(  C_{G}\left(  H\right)  \right)  =C_{G}\left(  C_{G}\left(  H\right)
\right)  \cap C_{G}\left(  H\right)  =C_{G}\left(  C_{G}\left(  H\right)
\right)  .
\]

\end{proof}

\bigskip

Using this lemma we can prove the first main result of this subsection:

\bigskip

\begin{teo}
\label{3.1.2}Let $G$ be a group and $H\leq G.$ If $H$ is $k$-nilpotent
subgroup, then the envelope $E_{k}$ is also $k$-nilpotent.
\end{teo}

\begin{proof}
From the second isomorphism theorem it can be written
\begin{equation}
HZ_{k-1}\left(  E_{k-1}\right)  \diagup Z_{k-1}\left(  E_{k-1}\right)  \cong
H\diagup H\cap Z_{k-1}\left(  E_{k-1}\right)  .\label{AC1}%
\end{equation}
Considering Lemma \ref{2.4}, we have%
\begin{equation}
Z_{k-1}\left(  H\right)  =Z_{k-1}\left(  E_{k-1}\right)  \cap H.\label{AC2}%
\end{equation}
If equation \ref{AC2} is used in equation \ref{AC1}, then
\[
HZ_{k-1}\left(  E_{k-1}\right)  \diagup Z_{k-1}\left(  E_{k-1}\right)  \cong
H\diagup Z_{k-1}\left(  H\right)  .
\]
Since $H$ is a $k$-nilpotent subgroup, $H\diagup Z_{k-1}\left(  H\right)  $ is
abelian. In a similar way, from the second isomorphism theorem we have
\begin{align*}
E_{k}Z_{k-1}\left(  E_{k-1}\right)  \diagup Z_{k-1}\left(  E_{k-1}\right)   &
\cong E_{k}\diagup E_{k}\cap Z_{k-1}\left(  E_{k-1}\right) \\
& =E_{k}\diagup Z_{k-1}\left(  E_{k}\right)  .
\end{align*}
From the definitions of $E_{k}$ and $C_{E_{k-1}}^{k}\left(  H\right)  $ we
write%
\[
E_{k}Z_{k-1}\left(  E_{k-1}\right)  \diagup Z_{k-1}\left(  E_{k-1}\right)
=C_{E_{k-1}\diagup Z_{k-1}\left(  E_{k-1}\right)  }\left(  C_{E_{k-1}\diagup
Z_{k-1}\left(  E_{k-1}\right)  }\left(  H\diagup Z_{k-1}\left(  H\right)
\right)  \right)  .
\]
Since $H\diagup Z_{k-1}\left(  H\right)  $ is abelian, the double centralizer
of this group is also abelian by Lemma \ref{3.1.1}. Then,
\begin{align*}
E_{k}Z_{k-1}\left(  E_{k-1}\right)  \diagup Z_{k-1}\left(  E_{k-1}\right)   &
=C_{E_{k-1}\diagup Z_{k-1}\left(  E_{k-1}\right)  }\left(  C_{E_{k-1}\diagup
Z_{k-1}\left(  E_{k-1}\right)  }\left(  H\diagup Z_{k-1}\left(  H\right)
\right)  \right) \\
& =E_{k}\diagup Z_{k-1}\left(  E_{k}\right)  .
\end{align*}
It means $E_{k}\diagup Z_{k-1}\left(  E_{k}\right)  $ is abelian. Thus $E_{k}$
is a $k$-nilpotent group.
\end{proof}

\bigskip

As a result of\textbf{\ }Theorem \ref{3.1.2}\textbf{\ }we verify the following
corollary which guarantees that the descending chain of envelopes stabilizes
at most at step the nilpotency class of $H.$

\begin{cor}
Let $G$ be a group and $H\leq G.$ If $H$ is $k$-nilpotent subgroup,
$E_{l}=E_{k}$ for all $l\geq k$ natural numbers.
\end{cor}

\begin{proof}
By Theorem \ref{3.1.2} we know that $E_{k}$ is a $k$-nilpotent subgroup. We
will argue by induction on $l.$ For $l=k+1,$%
\begin{align*}
C_{E_{k}}^{k+1}\left(  H\right)   & =\left\{  x\in E_{k}\mid\left[
x,H\right]  \subseteq C_{E_{k}}^{k}\left(  H\right)  \right\} \\
& =\left\{  x\in E_{k}\mid\left[  x,H\right]  \subseteq Z_{k}\left(
E_{k}\right)  \right\} \\
& =\left\{  x\in E_{k}\mid\left[  x,H\right]  \subseteq E_{k}\right\}  =E_{k}%
\end{align*}
where Lemma \ref{2.6} and $k$-nilpotence of $E_{k}$ are used. Since
\begin{align*}
E_{k+1}  & =\left\{  g\in E_{k}\mid\left[  g,C_{E_{k}}^{k+1}\left(  H\right)
\right]  \leq C_{E_{k}}^{k}\left(  H\right)  \right\} \\
& =\left\{  g\in E_{k}\mid\left[  g,E_{k}\right]  \leq E_{k}\right\}  =E_{k}%
\end{align*}
claim is true for $l=k+1.$ Suppose that our claim is true for $l=k+n,$ i.e.
$E_{k+n}=E_{k},$ we shall verify the same equality for$E_{k+n+1}.$ Since
\[
E_{k}=Z_{k}\left(  E_{k}\right)  =C_{E_{k}}^{k}\left(  H\right)  \leq
C_{E_{k}}^{k+n}\left(  H\right)  \leq E_{k}%
\]
$C_{E_{k}}^{k+n}\left(  H\right)  =E_{k}$. Then, by induction we have
\begin{align*}
C_{E_{k+n}}^{k+n+1}\left(  H\right)   & =\left\{  x\in\underset{i\leq
k+n}{\cap}N_{E_{k+n}}\left(  C_{E_{k+n}}^{i}\left(  H\right)  \right)
\mid\left[  x,H\right]  \subseteq C_{E_{k+n}}^{k+n}\left(  H\right)  \right\}
\\
& =\left\{  x\in\underset{i\leq k+n}{\cap}N_{E_{k+n}}\left(  C_{E_{k+n}}%
^{i}\left(  H\right)  \right)  \mid\left[  x,H\right]  \subseteq C_{E_{k}%
}^{k+n}\left(  H\right)  \right\} \\
& =\left\{  x\in E_{k}\mid\left[  x,H\right]  \subseteq E_{k}\right\}  =E_{k}.
\end{align*}
Then, the following equation is obtained:\
\begin{align*}
E_{k+n+1}  & =\left\{  g\in E_{k+n}\mid\left[  g,C_{E_{k+n}}^{k+n+1}\left(
H\right)  \right]  \leq C_{E_{k+n}}^{k+n}\left(  H\right)  \right\} \\
& =\left\{  g\in E_{k}\mid\left[  g,E_{k}\right]  \leq C_{E_{k}}^{k+n}\left(
H\right)  \right\}  =\left\{  g\in E_{k}\mid\left[  g,E_{k}\right]  \leq
E_{k}\right\}  =E_{k}.
\end{align*}
Thus the claim holds for all $l\geq k$ natural numbers.
\end{proof}

\bigskip

In a separate preprint, we generalize the results of this subsection to
hypercentral subgroups by defining transfinite versions of the $E_{k}$-chains
(\cite{Cak}).

\bigskip

\subsection{Topological Results\label{3.2}}

In this subsection, we will be working in a group $G$ endowed with a topology
where singletons are closed and the following functions are continuous
\[
x\rightarrow x^{-1},\text{ \ \ }x\rightarrow ax,\text{ \ \ }x\rightarrow
xa,\text{ \ \ }x\rightarrow x^{-1}ax
\]
for all $a\in G$. We will show that the presence of such a topology on $G$ is
sufficient to ensure that the $E_{k}$-envelopes of arbitrary subgroups are
closed. In classes of groups that satisfy noetherianity properties of closed
subgroups, this result suffices to conclude that the $E_{k}$-envelopes
stabilize. An example of such class is that of groups satisfying the chain
condition on closed subgroups (see\textbf{\ }\cite{Bryant 2}\textbf{).}

\bigskip

\begin{lem}
\label{3.2.1}Let $G$ be a group satisfying the standing topological hypothesis
and $X\subseteq G$. Then, for every $i\in%
%TCIMACRO{\U{2115} }%
%BeginExpansion
\mathbb{N}
%EndExpansion
,$ the subgroup $C_{G}^{i}\left(  X\right)  $ is closed.
\end{lem}

\begin{proof}
We proceed by induction on $i$. When $i=0$, $C_{G}^{0}\left(  X\right)
=\left\{  e\right\}  ,$ and since each single element subset of $G$ is closed,
the claim holds. We now assume that $C_{G}^{j}\left(  X\right)  $ iterated
centralizers are closed for all $j$ natural numbers such that $j<i.$ By
definition,
\[
\bigskip C_{G}^{i}\left(  X\right)  =\left\{  g\in\underset{j<i}{\cap}%
N_{G}\left(  C_{G}^{j}\left(  X\right)  \right)  \mid\left[  g,X\right]  \leq
C_{G}^{i-1}\left(  X\right)  \right\}  .
\]
Each $C_{G}^{j}\left(  X\right)  $ is closed by the inductive assumption for
$j<i$. Then by \cite[Lemma 5.4]{W}, $N_{G}\left(  C_{G}^{j}\left(  X\right)
\right)  $ is closed. Hence $\underset{j<i}{\cap}N_{G}\left(  C_{G}^{j}\left(
X\right)  \right)  $ is closed. On the other hand, the following function is
continuous in $G$%
\[
k_{x}:G\rightarrow G,\ \ \ g\longmapsto\left[  g,x\right]
\]
where $x$ is a fixed element of $X$. The inverse image of $C_{G}^{i-1}\left(
X\right)  \leq G$ with respect to function $k_{x}$ is%
\[
k_{x}^{-1}\left(  C_{G}^{i-1}\left(  X\right)  \right)  =\left\{  g\in G\mid
k_{x}\left(  g\right)  \in C_{G}^{i-1}\left(  X\right)  \right\}  =\left\{
g\in G\mid\left[  g,x\right]  \in C_{G}^{i-1}\left(  X\right)  \right\}  .
\]
As the intersection $\cap k_{x}^{-1}\left(  C_{G}^{i-1}\left(  X\right)
\right)  $ is also closed, and%
\begin{align*}
g  & \in\underset{x\in X}{\cap}k_{x}^{-1}\left(  C_{G}^{i-1}\left(  X\right)
\right) \\
& \Leftrightarrow\left[  g,x\right]  \in C_{G}^{i-1}\left(  X\right)  ,\text{
\ for all }x\in X\\
& \Leftrightarrow\left[  g,X\right]  \subset C_{G}^{i-1}\left(  X\right)  ,
\end{align*}
it follows that $C_{G}^{i}\left(  X\right)  $ is closed.
\end{proof}

\bigskip

\begin{lem}
Let $G$ be a group satisfying the topological hypothesis of this subsection
and $H$ a subgroup of $G$. Then the $E_{k}\left(  H\right)  $ are closed.
\end{lem}

\begin{proof}
Since $G$ is the ambient space, our claim is trivial for $k=0$. We now assume
that $E_{k}$ is closed. For $k+1$, we have
\begin{align*}
E_{k+1}  & =\left\{  g\in E_{k}\mid\left[  g,C_{E_{k}}^{k+1}\left(  H\right)
\right]  \leq C_{E_{k}}^{k}\left(  H\right)  \right\} \\
& =\left\{  g\in E_{k}\mid\left[  g,C_{E_{k}}^{k+1}\left(  H\right)  \right]
\leq Z_{k}\left(  E_{k}\right)  \right\}  .
\end{align*}
For iterated centers $Z_{k}\left(  E_{k}\right)  =C_{G}^{k}\left(
E_{k}\right)  \cap E_{k}$. By the induction hypothesis and Lemma
\ref{3.2.1}\textbf{,} $E_{k}$\textbf{\ }and $C_{G}^{k}\left(  E_{k}\right)  $
are closed. Thus, $Z_{k}(E_{k})$ is also closed. Let $x\in C_{E_{k}(H)}%
^{k+1}\left(  H\right)  .$ By the properties of the topology on $G$, the
following function is continuous:%
\[
k_{x}:G\rightarrow G,\ \ \ g\longmapsto\left[  g,x\right]  .
\]
The inverse image of $Z_{k}\left(  E_{k}\right)  $ with respect to $k_{x}$ is
\[
k_{x}^{-1}\left(  Z_{k}\left(  E_{k}\right)  \right)  =\left\{  g\in G\mid
k_{x}\left(  g\right)  \in Z_{k}\left(  E_{k}\right)  \right\}  =\left\{  g\in
G\mid\left[  g,x\right]  \in Z_{k}\left(  E_{k}\right)  \right\}
\]
and closed. Moreover,
\begin{align*}
g  & \in\cap k_{x}^{-1}\left(  Z_{k}\left(  E_{k}\right)  \right)
\Leftrightarrow g\in k_{x}^{-1}\left(  Z_{k}\left(  E_{k}\right)  \right)
,\text{ \ for all }x\in C_{E_{k}(H)}^{k+1}\left(  H\right) \\
& \Leftrightarrow\left[  g,x\right]  \in Z_{k}\left(  E_{k}\right)  ,\text{
\ for all }x\in C_{E_{k}(H)}^{k+1}\left(  H\right) \\
& \Leftrightarrow\left[  g,C_{E_{k}(H)}^{k+1}\left(  H\right)  \right]
\subset Z_{k}\left(  E_{k}\right)
\end{align*}
It follows that $E_{k+1}$ is closed.
\end{proof}

\bigskip

\begin{cor}
Let $G$ be a group which satisfies the standing topological hypothesis of the
subsection and $H\leq G$. If $G$ satisfies the minimal condition on closed
subgroups then $E_{k}(H)$- envelopes stabilize.
\end{cor}

\bigskip

\begin{cor}
Let $G$ be a linear group and $H\leq G$. Then $E_{k}(H)$-envelopes stabilize
in $G$.
\end{cor}

\begin{proof}
By Theorem 3.5 of \cite{Bryant 2},\textbf{\ }$G$ satisfies the chain condition
on closed subgroups.
\end{proof}

\bigskip

\section{Counter-Example}

\bigskip In this section we will construct a counterexample to the
stabilization problem of the $E_{k}$-envelopes. We will be working in the
symmetric group on the natural numbers $Sym\left(
%TCIMACRO{\U{2115} }%
%BeginExpansion
\mathbb{N}
%EndExpansion
\right)  $ that we will denote $G$. The subgroup $H$ whose $E_{k}$-envelopes
form an infinite descending chain is defined as follows. Let $K=\underset{x\in%
%TCIMACRO{\U{2115} }%
%BeginExpansion
\mathbb{N}
%EndExpansion
}{\oplus}\left\langle \left(  2x\text{ \ }2x+1\right)  \right\rangle $. Note
that $K$ is a normal subgroup of $G.$ Now we define a special permutation by
the action $\left(  2x\text{\ \ }2x+1\right)  \mapsto\left(  2f\left(
x\right)  \text{\ \ }2f\left(  x\right)  +1\right)  $ for every $x\in%
%TCIMACRO{\U{2115} }%
%BeginExpansion
\mathbb{N}
%EndExpansion
,$ where%

\[
f\left(  x\right)  =\left\{
\begin{array}
[c]{c}%
x\longmapsto x+2,\text{\ \ \ \ \ \ \ \ \ \ \ \ \ \ \ \ \ \ \ if }x\text{ is
even,}\\
x\longmapsto x-2,\text{\ \ \ \ \ \ if }x\text{ is odd and }x\neq1\text{,}\\
1\longmapsto0,\text{
\ \ \ \ \ \ \ \ \ \ \ \ \ \ \ \ \ \ \ \ \ \ \ \ \ \ \ \ if }x=1\text{ }%
\end{array}
\right.  .
\]

\bigskip We define $H=K\rtimes\langle f\rangle.$ We will show the
$E_{k}\left(  H\right)  $ form an infinite descending chain.

We emphasize that $G$ is not an $\mathfrak{M}_{C}$-group. Indeed, one can
easily show the existence of infinite descending chains of centralizers.

From now until the end of the section we will detail the determination of the
nature of $E_{k}\left(  H\right)  .$ The main step will be to show that the
iterated centralizers of $H$ are finite subgroups of $G$ whose nontrivial
elements are of infinite support and that form an infinite ascending chain.

\bigskip

We use a special notation for the elements of the group $\underset{x\in%
%TCIMACRO{\U{2115} }%
%BeginExpansion
\mathbb{N}
%EndExpansion
}{\Pi}\left\langle \left(  2x\text{ \ }2x+1\right)  \right\rangle $:%
\[
g\in\underset{x\in%
%TCIMACRO{\U{2115} }%
%BeginExpansion
\mathbb{N}
%EndExpansion
}{\Pi}\left\langle \left(  2x\text{ \ }2x+1\right)  \right\rangle \Rightarrow
g=\underset{x\in%
%TCIMACRO{\U{2115} }%
%BeginExpansion
\mathbb{N}
%EndExpansion
}{\Pi}\left(  2x\text{ \ }2x+1\right)  ^{j_{g}\left(  x\right)  }%
\]
where $j_{g}(x)\in\{0,1\}.$ The arguments about the function $j_{g}$ involve
elementary arithmetic. This will always be modulo 2.

The actions on $%
%TCIMACRO{\U{2115} }%
%BeginExpansion
\mathbb{N}
%EndExpansion
$ will be on the left

\bigskip

First, we shall give a technical lemma which include two known results:

\begin{lemma}
\label{4.1}Let $G_{0}$ satisfy the property
\[
G\geq G_{0}\geq\underset{x\in%
%TCIMACRO{\U{2115} }%
%BeginExpansion
\mathbb{N}
%EndExpansion
}{\Pi}\left\langle \left(  2x\text{ \ }2x+1\right)  \right\rangle .
\]
Then the following equalities hold:

\begin{enumerate}
\item[(i)] $C_{G}\left(  \underset{x\in%
%TCIMACRO{\U{2115} }%
%BeginExpansion
\mathbb{N}
%EndExpansion
}{\oplus}\left\langle \left(  2x\text{ \ }2x+1\right)  \right\rangle \right)
=\underset{x\in%
%TCIMACRO{\U{2115} }%
%BeginExpansion
\mathbb{N}
%EndExpansion
}{\prod}\left\langle \left(  2x\text{ \ }2x+1\right)  \right\rangle ,$

\item[(ii)] $C_{G_{0}}\left(  H\right)  =\left\langle \underset{x\in%
%TCIMACRO{\U{2115} }%
%BeginExpansion
\mathbb{N}
%EndExpansion
}{\prod}\left(  2x\text{ \ }2x+1\right)  \right\rangle .$
\end{enumerate}
\end{lemma}

\begin{proof}

\begin{enumerate}
\item[(i)] Since $C_{G}\left(  \underset{x\in%
%TCIMACRO{\U{2115} }%
%BeginExpansion
\mathbb{N}
%EndExpansion
}{\oplus}\left\langle \left(  2x\text{ \ }2x+1\right)  \right\rangle \right)
=C_{G}\left(  \left\{  \left(  2x\text{ \ }2x+1\right)  \mid x\in%
%TCIMACRO{\U{2115} }%
%BeginExpansion
\mathbb{N}
%EndExpansion
\right\}  \right)  $, $g\in C_{G}\left(  \underset{x\in%
%TCIMACRO{\U{2115} }%
%BeginExpansion
\mathbb{N}
%EndExpansion
}{\oplus}\left\langle \left(  2x\text{ \ }2x+1\right)  \right\rangle \right)
$ if and only if $g\left(  2x\text{ \ }2x+1\right)  g^{-1}=\left(  2x\text{
\ }2x+1\right)  $ for every $x\in%
%TCIMACRO{\U{2115} }%
%BeginExpansion
\mathbb{N}
%EndExpansion
$ if and only if $g\left\{  2x,2x+1\right\}  =\left\{  2x,2x+1\right\}  $ for
every $x\in%
%TCIMACRO{\U{2115} }%
%BeginExpansion
\mathbb{N}
%EndExpansion
$.

\item[(ii)] Since the centralizer of $K$ in $G_{0}$ is known by (i), computing
the centralizer $C_{C_{G_{0}}\left(  K\right)  }\left(  \left\langle
f\right\rangle \right)  $ is enough. Let $g\in C_{G_{0}}\left(  K\right)  .$
Then by part (i), $g$ is of the form $\underset{x\in%
%TCIMACRO{\U{2115} }%
%BeginExpansion
\mathbb{N}
%EndExpansion
}{\prod}\left(  2x\text{ \ }2x+1\right)  ^{j_{g}\left(  x\right)  }$,
$j_{g}\left(  x\right)  \in\left\{  0,1\right\}  .$ According to this
\begin{align*}
\left[  g,f\right]   & =\left[  \underset{x\in%
%TCIMACRO{\U{2115} }%
%BeginExpansion
\mathbb{N}
%EndExpansion
}{\prod}\left(  2x\text{ \ }2x+1\right)  ^{j_{g}\left(  x\right)  },f\right]
\\
& =\underset{x\in%
%TCIMACRO{\U{2115} }%
%BeginExpansion
\mathbb{N}
%EndExpansion
}{\prod}\left(  \text{ }\left(  2x\text{ \ }2x+1\right)  ^{j_{g}\left(
x\right)  }f^{-1}\left(  2x\text{ \ }2x+1\right)  ^{j_{g}\left(  x\right)
}f\text{ }\right) \\
& =\underset{x\in%
%TCIMACRO{\U{2115} }%
%BeginExpansion
\mathbb{N}
%EndExpansion
}{\prod}\left(  \text{ }\left(  2x\text{ \ }2x+1\right)  ^{j_{g}\left(
x\right)  }\left(  2f^{-1}\left(  x\right)  \text{ \ }2f^{-1}\left(  x\right)
+1\right)  ^{j_{g}\left(  x\right)  }\text{ }\right) \\
& =\underset{x\in%
%TCIMACRO{\U{2115} }%
%BeginExpansion
\mathbb{N}
%EndExpansion
}{\prod}\left(  2x\text{ \ }2x+1\right)  ^{j_{g}\left(  x\right)
+j_{g}\left(  f\left(  x\right)  \right)  }%
\end{align*}
is obtained. Then we have
\begin{align*}
\left[  g,f\right]   & =1\Leftrightarrow\underset{x\in%
%TCIMACRO{\U{2115} }%
%BeginExpansion
\mathbb{N}
%EndExpansion
}{\prod}\left(  2x\text{ \ }2x+1\right)  ^{j_{g}\left(  x\right)
+j_{g}\left(  f\left(  x\right)  \right)  }=1\\
& \Leftrightarrow j_{g}\left(  x\right)  +j_{g}\left(  f\left(  x\right)
\right)  =0\text{ \ \ for all }x\in%
%TCIMACRO{\U{2115}}%
%BeginExpansion
\mathbb{N}%
%EndExpansion
\\
& \Leftrightarrow j_{g}\left(  x\right)  =j_{g}\left(  f\left(  x\right)
\right)  \text{ \ \ \ \ \ \ \ for all }x\in%
%TCIMACRO{\U{2115} }%
%BeginExpansion
\mathbb{N}
%EndExpansion
.
\end{align*}
It follows that
\[
C_{C_{G_{0}}\left(  K\right)  }\left(  \left\langle f\right\rangle \right)
=\left\{  1,\text{ }\underset{x\in%
%TCIMACRO{\U{2115} }%
%BeginExpansion
\mathbb{N}
%EndExpansion
}{\prod}\left(  2x\text{ \ }2x+1\right)  \right\}  =\left\langle
\underset{x\in%
%TCIMACRO{\U{2115} }%
%BeginExpansion
\mathbb{N}
%EndExpansion
}{\prod}\left(  2x\text{ \ }2x+1\right)  \right\rangle ,
\]
and $C_{G_{0}}\left(  H\right)  =\left\langle \underset{x\in%
%TCIMACRO{\U{2115} }%
%BeginExpansion
\mathbb{N}
%EndExpansion
}{\prod}\left(  2x\text{ \ }2x+1\right)  \right\rangle .$
\end{enumerate}
\end{proof}

\bigskip

The following result will be used in computing the iterated centralizers. It
will first be proven under a specific hypothesis which will be eliminated in
Corollary \ref{4.3}.

\bigskip

\begin{proposition}
\label{4.2}Suppose that for every $k\in%
%TCIMACRO{\U{2115} }%
%BeginExpansion
\mathbb{N}
%EndExpansion
,$ $\underset{x\in%
%TCIMACRO{\U{2115} }%
%BeginExpansion
\mathbb{N}
%EndExpansion
}{\prod}\left\langle \left(  2x\text{ \ }2x+1\right)  \right\rangle \leq
E_{k}\left(  H\right)  $. Then $C_{E_{k}\left(  H\right)  }^{i+1}\left(
H\right)  \leq\underset{x\in%
%TCIMACRO{\U{2115} }%
%BeginExpansion
\mathbb{N}
%EndExpansion
}{\prod}\left\langle \left(  2x\text{ \ }2x+1\right)  \right\rangle $ for all
$k\in%
%TCIMACRO{\U{2115} }%
%BeginExpansion
\mathbb{N}
%EndExpansion
$ and $i\leq k$, and for every $h\in C_{E_{k}\left(  H\right)  }^{i+1}\left(
H\right)  $, for all $x\in%
%TCIMACRO{\U{2115} }%
%BeginExpansion
\mathbb{N}
%EndExpansion
$ the following relation holds:
\[
j_{h}\left(  x\right)  =j_{h}\left(  x+2^{\left(  i+1\right)  }\right)  .
\]
In particular, $C_{E_{k}\left(  H\right)  }^{i+1}\left(  H\right)  $ is finite
for every $k\in%
%TCIMACRO{\U{2115} }%
%BeginExpansion
\mathbb{N}
%EndExpansion
$ and $i\leq k+1.$
\end{proposition}

\begin{proof}
Let $k\in%
%TCIMACRO{\U{2115} }%
%BeginExpansion
\mathbb{N}
%EndExpansion
.$ We proceed by induction on $i\leq k$. For $i=0$, we want to show that claim
holds in the centralizer $C_{E_{k}\left(  H\right)  }\left(  H\right)  $. It
is enough to compute the centralizers $C_{E_{k}\left(  H\right)  }\left(
K\right)  $ and $C_{E_{k}\left(  H\right)  }\left(  \left\langle
f\right\rangle \right)  $. By the hypothesis of the proposition, since the
group $\underset{x\in%
%TCIMACRO{\U{2115} }%
%BeginExpansion
\mathbb{N}
%EndExpansion
}{\prod}\left\langle \left(  2x\text{ \ }2x+1\right)  \right\rangle $ is
abelian, $\left\langle \underset{x\in%
%TCIMACRO{\U{2115} }%
%BeginExpansion
\mathbb{N}
%EndExpansion
}{\prod}\left(  2x\text{ \ }2x+1\right)  \right\rangle \leq$ $C_{E_{k}\left(
H\right)  }\left(  H\right)  $. Now we show that this inclusion in fact is an equality.

By Lemma \ref{4.1}\textit{\ (i)}\textbf{\ }we have $C_{E_{k}\left(  H\right)
}\left(  K\right)  =\underset{x\in%
%TCIMACRO{\U{2115} }%
%BeginExpansion
\mathbb{N}
%EndExpansion
}{\prod}\left\langle \left(  2x\text{ \ }2x+1\right)  \right\rangle .$We now
compute $C_{C_{E_{k}\left(  H\right)  }\left(  K\right)  }\left(  \left\langle
f\right\rangle \right)  $. Let $g\in\underset{x\in%
%TCIMACRO{\U{2115} }%
%BeginExpansion
\mathbb{N}
%EndExpansion
}{\prod}\left\langle \left(  2x\text{ \ }2x+1\right)  \right\rangle .$ Thus $g
$ is the form $\underset{x\in%
%TCIMACRO{\U{2115} }%
%BeginExpansion
\mathbb{N}
%EndExpansion
}{\prod}\left(  2x\text{ \ }2x+1\right)  ^{j_{g}\left(  x\right)  }$. We have
already proven in Lemma \ref{4.1} that $\left[  g,f\right]  =1$ if and only if
$j_{g}\left(  x\right)  =j_{g}\left(  f\left(  x\right)  \right)  $ for every
$x\in%
%TCIMACRO{\U{2115} }%
%BeginExpansion
\mathbb{N}
%EndExpansion
.$

Taking into consideration the function $f$,
\[
j_{g}\left(  x\right)  =j_{g}\left(  f\left(  x\right)  \right)
\Leftrightarrow\left\{
\begin{array}
[c]{c}%
j_{g}\left(  2x\right)  =j_{g}\left(  2x+2\right)  ,\text{ \ \ \ \ \ \ }x\in%
%TCIMACRO{\U{2115}}%
%BeginExpansion
\mathbb{N}%
%EndExpansion
\\
j_{g}\left(  2x+1\right)  =j_{g}\left(  2x-1\right)  ,\text{ \ }1\leq x,\text{
}x\in%
%TCIMACRO{\U{2115}}%
%BeginExpansion
\mathbb{N}%
%EndExpansion
\\
j_{g}\left(  1\right)  =j_{g}\left(  0\right)  \text{
\ \ \ \ \ \ \ \ \ \ \ \ \ \ \ \ \ \ \ \ }%
\end{array}
\right.
\]
is obtained. So,
\[
C_{C_{E_{k}\left(  H\right)  }\left(  K\right)  }\left(  \left\langle
f\right\rangle \right)  =\left\{  1,\text{ }\underset{x\in%
%TCIMACRO{\U{2115} }%
%BeginExpansion
\mathbb{N}
%EndExpansion
}{\prod}\left(  2x\text{ \ }2x+1\right)  \right\}  =\left\langle
\underset{x\in%
%TCIMACRO{\U{2115} }%
%BeginExpansion
\mathbb{N}
%EndExpansion
}{\prod}\left(  2x\text{ \ }2x+1\right)  \right\rangle .
\]
Therefore,
\[
C_{E_{k}\left(  H\right)  }\left(  H\right)  =\left\langle \underset{x\in%
%TCIMACRO{\U{2115} }%
%BeginExpansion
\mathbb{N}
%EndExpansion
}{\prod}\left(  2x\text{ \ }2x+1\right)  \right\rangle .
\]
The equation $j_{g}\left(  x\right)  =j_{g}\left(  x+2\right)  $ holds
directly since the equation $j_{g}\left(  x\right)  =j_{g}\left(  x+1\right)
$ is true in $C_{E_{k}\left(  H\right)  }\left(  H\right)  $. The claim
follows for $i=0$.

Assume it is satisfied for $\left(  i-1\right)  .$ So by induction
$C_{E_{k}\left(  H\right)  }^{i}\left(  H\right)  \leq\underset{x\in%
%TCIMACRO{\U{2115} }%
%BeginExpansion
\mathbb{N}
%EndExpansion
}{\prod}\left\langle \left(  2x\text{ \ }2x+1\right)  \right\rangle $ and for
every $h\in C_{E_{k}\left(  H\right)  }^{i}\left(  H\right)  ,$ the equality
$j_{h}\left(  x\right)  =j_{h}\left(  x+2^{i}\right)  $ holds for every $x\in%
%TCIMACRO{\U{2115} }%
%BeginExpansion
\mathbb{N}
%EndExpansion
$. Note that it follows from the induction hypothesis and the equality
$j_{h}\left(  x\right)  =j_{h}\left(  x+2^{i}\right)  $ that the elements of
$C_{E_{k}\left(  H\right)  }^{i}\left(  H\right)  $ apart from the identity
element have infinite support. For $i\in%
%TCIMACRO{\U{2115} }%
%BeginExpansion
\mathbb{N}
%EndExpansion
$, since $C_{E_{k}\left(  H\right)  }^{i+1}\left(  H\right)  =\left\{  g\in
E_{k}\left(  H\right)  \mid\left[  g,H\right]  \subseteq C_{E_{k}\left(
H\right)  }^{i}\left(  H\right)  \right\}  $ and $H=K\rtimes\left\langle
f\right\rangle ,$ every $g\in C_{E_{k}\left(  H\right)  }^{i+1}\left(
H\right)  $ has no satisfy the following conditions:

\begin{itemize}
\item $\left[  g,K\right]  \subseteq C_{E_{k}\left(  H\right)  }^{i}\left(
H\right)  ,$

\item $\left[  g,f\right]  \subseteq C_{E_{k}\left(  H\right)  }^{i}\left(
H\right)  .$
\end{itemize}

Our standing hypothesis on $\underset{x\in%
%TCIMACRO{\U{2115} }%
%BeginExpansion
\mathbb{N}
%EndExpansion
}{\prod}\left\langle \left(  2x\text{ \ }2x+1\right)  \right\rangle $ yields
two cases:

\begin{description}
\item[a] $g\in\underset{x\in%
%TCIMACRO{\U{2115} }%
%BeginExpansion
\mathbb{N}
%EndExpansion
}{\prod}\left\langle \left(  2x\text{ \ }2x+1\right)  \right\rangle ,$

\item[b] $g\in E_{k}\left(  H\right)  \diagdown\underset{x\in%
%TCIMACRO{\U{2115} }%
%BeginExpansion
\mathbb{N}
%EndExpansion
}{\prod}\left\langle \left(  2x\text{ \ }2x+1\right)  \right\rangle .$
\end{description}

We will describe the form of $g\in E_{k}\left(  H\right)  $ satisfying these conditions.

\begin{itemize}
\item We start with the commutator condition on $\left[  g,K\right]  $.
\end{itemize}

\begin{description}
\item[a] For $g\in\underset{x\in%
%TCIMACRO{\U{2115} }%
%BeginExpansion
\mathbb{N}
%EndExpansion
}{\prod}\left\langle \left(  2x\text{ \ }2x+1\right)  \right\rangle $,
$\left[  g,K\right]  =1\subseteq C_{E_{k}\left(  H\right)  }^{i}\left(
H\right)  $.

\item[b] For $g\in E_{k}\left(  H\right)  \diagdown\underset{x\in%
%TCIMACRO{\U{2115} }%
%BeginExpansion
\mathbb{N}
%EndExpansion
}{\prod}\left\langle \left(  2x\text{ \ }2x+1\right)  \right\rangle $, there
exists $x\in%
%TCIMACRO{\U{2115} }%
%BeginExpansion
\mathbb{N}
%EndExpansion
$ such that
\[
g^{-1}\left\{  2x,2x+1\right\}  \neq\left\{  2x,2x+1\right\}  .
\]
Thus $\left[  g,\left(  2x\text{ \ }2x+1\right)  \right]  \neq1.$ But $\left(
2x\text{ \ }2x+1\right)  \in K$ and one can easily compute that $\left[
g,\left(  2x\text{ \ }2x+1\right)  \right]  $ has finite support. Hence we
have found an element of $K$, namely $\left(  2x\text{ \ }2x+1\right)  $, such
that $\left[  g,\left(  2x\text{ \ }2x+1\right)  \right]  \notin
C_{E_{k}\left(  H\right)  }^{i}\left(  H\right)  $; indeed, as it was
mentioned at the beginning of the inductive step of the proof, all nontrivial
elements of $C_{E_{k}\left(  H\right)  }^{i}\left(  H\right)  $ are of
infinite support. Thus, $C_{E_{k}\left(  H\right)  }^{i+1}\left(  H\right)
\leq\underset{x\in%
%TCIMACRO{\U{2115} }%
%BeginExpansion
\mathbb{N}
%EndExpansion
}{\prod}\left\langle \left(  2x\text{ \ }2x+1\right)  \right\rangle .$\ \ \ \ \ \ \ \ \ \ \ \ \ \ \ \ \ \ \ \ \ \ \ \ \ \ \ \ \ \ \ \ \ \ \ \ \ \ \ \ \ \ \ \ \ \ \ \ \ \ \ \ \ \ \ \ \ \ \ \ \ \ \ \ \ \ \ \ \ \ \ \ \ \ \ \ \ \ \ \ \ \ \ \ \ \ \ \ \ \ \ \ \ \ \ \ \ \ \ \ \ \ \ \ \ \ \ \ \ \ \ \ \ \ \ \ \ \ \ \ \ \ \ \ \ \ \ \ \ \ \ \ \ \ \ \ \ \ \ \ \ \ \ \ \ \ \ \ \ \ \ \ \ \ \ \ \ \ \ \ \ \ \ \ \ \ \ \ \ \ \ \ \ \ \ \ \ \ \ \ \ \ \ \ \ \ \ \ \ \ \ \ \ \ \ \ \ \ \ \ \ \ \ \ \ \ \ \ \ \ \ \ \ \ \ \ \ \ \ \ \ \ \ \ \ \ \ \ \ \ \ \ \ \ \ \ \ \ \ \ \ \ \ \ \ \ \ \ \ \ \ \ \ \ \ \ \ \ \ \ \ \ \ \ \ \ \ \ \ \ \ \ \ \ \ \ \ \ \ \ \ \ \ \ \ \ \ \ \ \ \ \ \ \ \ \ \ \ \ \ \ \ \ \ \ \ \ \ \ \ \ \ \ \ \ \ \ \ \ \ \ \ \ \ \ \ \ \ \ \ \ \ \ \ \ \ \ \ \ \ \ \ \ \ \ \ \ \ \ \ \ \ \ \ \ \ \ \ \ \ \ \ \ \ \ \ \ \ \ \ \ \ \ \ \ \ \ \ \ \ \ \ \ \ \ \ \ \ \ \ \ \ \ \ \ \ \ \ \ \ \ \ \ \ \ \ \ \ \ \ \ \ \ \ \ \ \ \ \ \ \ \ \ \ \ \ \ \ \ \ \ \ \ \ \ \ \ \ \ \ \ \ \ \ \ \ \ \ \ \ \ \ \ \ \ \ \ \ \ \ \ \ \ \ \ \ \ \ \ \ \ \ \ \ \ \ \ \ \ \ \ \ \ \ \ \ \ \ \ \ \ \ \ \ \ \ \ \ \ \ \ \ \ \ \ \ \ \ \ \ \ \ \ \ \ \ \ \ \ \ \ \ \ \ \ \ \ \ \ \ \ \ \ \ \ \ \ \ \ \ \ \ \ \ \ \ \ \ \ \ \ \ \ \ \ \ \ \ \ \ \ \ \ \ \ \ \ \ \ \ \ \ \ \ \ \ \ \ \ \ \ \ \ \ \ \ \ \ \ \ \ \ \ \ \ \ \ \ \ \ \ \ \ \ \ \ \ \ \ \ \ \ \ \ \ \ \ \ \ \ \ \ \ \ \ \ \ \ \ \ \ \ \ \ \ \ \ \ \ \ \ \ \ \ \ \ \ \ \ \ \ \ \ \ \ \ \ \ \ \ \ \ \ \ \ \ \ \ \ \ \ \ \ \ \ \ 
\end{description}

\begin{itemize}
\item Using the previous step, we analyze $g\in\underset{x\in%
%TCIMACRO{\U{2115} }%
%BeginExpansion
\mathbb{N}
%EndExpansion
}{\prod}\left\langle \left(  2x\text{ \ }2x+1\right)  \right\rangle $ such
that $\left[  g,f\right]  \subseteq C_{E_{k}\left(  H\right)  }^{i}\left(
H\right)  $. Hence, $g$ is of the form $\underset{x\in%
%TCIMACRO{\U{2115} }%
%BeginExpansion
\mathbb{N}
%EndExpansion
}{\prod}\left(  2x\text{ \ }2x+1\right)  ^{j_{g}\left(  x\right)  }$. We have
already verified that $\left[  g,f\right]  =\underset{x\in%
%TCIMACRO{\U{2115} }%
%BeginExpansion
\mathbb{N}
%EndExpansion
}{\prod}\left(  2x\text{ \ }2x+1\right)  ^{j_{g}\left(  x\right)
+j_{g}\left(  f\left(  x\right)  \right)  }$ for every $x\in%
%TCIMACRO{\U{2115} }%
%BeginExpansion
\mathbb{N}
%EndExpansion
$. Since $g$ satisfies the condition
\[
\left[  g,f\right]  =\underset{x\in%
%TCIMACRO{\U{2115} }%
%BeginExpansion
\mathbb{N}
%EndExpansion
}{\prod}\left(  2x\text{ \ }2x+1\right)  ^{j_{g}\left(  x\right)
+j_{g}\left(  f\left(  x\right)  \right)  }\in C_{E_{k}\left(  H\right)  }%
^{i}\left(  H\right)  ,
\]
\ \ \ there exists $\underset{x\in%
%TCIMACRO{\U{2115} }%
%BeginExpansion
\mathbb{N}
%EndExpansion
}{h=\prod}\left(  2x\text{ \ }2x+1\right)  ^{j_{h}\left(  x\right)  }\in
C_{E_{k}\left(  H\right)  }^{i}\left(  H\right)  $ such that
\[
\underset{x\in%
%TCIMACRO{\U{2115} }%
%BeginExpansion
\mathbb{N}
%EndExpansion
}{\prod}\left(  2x\text{ \ }2x+1\right)  ^{j_{g}\left(  x\right)
+j_{g}\left(  f\left(  x\right)  \right)  }=\underset{x\in%
%TCIMACRO{\U{2115} }%
%BeginExpansion
\mathbb{N}
%EndExpansion
}{\prod}\left(  2x\text{ \ }2x+1\right)  ^{j_{h}\left(  x\right)  }%
\]
Thus $j_{g}\left(  x\right)  +j_{g}\left(  f\left(  x\right)  \right)
=j_{h}\left(  x\right)  $ for every $x\in%
%TCIMACRO{\U{2115} }%
%BeginExpansion
\mathbb{N}
%EndExpansion
$. By the definition of $f$, for $a\in%
%TCIMACRO{\U{2115} }%
%BeginExpansion
\mathbb{N}
%EndExpansion
$, the following relations hold:
\begin{equation}
\left\{
\begin{array}
[c]{c}%
j_{g}\left(  2a\right)  +j_{g}\left(  2a+2\right)  =j_{h}\left(  2a\right) \\
j_{g}\left(  2a+3\right)  +j_{g}\left(  2a+1\right)  =j_{h}\left(  2a+3\right)
\\
\text{ \ \ \ \ \ \ }j_{g}\left(  1\right)  +j_{g}\left(  0\right)
=j_{h}\left(  1\right)
\end{array}
\right. \label{AC3}%
\end{equation}

\end{itemize}

The powers $j_{g}\left(  x\right)  $ will be determined seperately but by
using similar methods according to the parity of $x$. By the induction
hypothesis the following system of equations is obtained for even $x$:%
\[
j_{g}\left(  2a\right)  +j_{g}\left(  2a+2\right)  =j_{h}\left(  2a\right)
,\text{ \ }0\leq a<2^{i-1}.
\]

Summing up the two sides of these equations, we get
\[
\text{\ }j_{g}\left(  0\right)  +2\text{\ }j_{g}\left(  2\right)
+...+j_{g}\left(  2^{i}\right)  =j_{h}\left(  0\right)  +...+j_{h}\left(
2^{i}-2\right)  ,
\]
and
\[
j_{g}\left(  0\right)  +\text{\ }j_{g}\left(  2^{i}\right)  =\sum
_{a=0}^{2^{i-1}-1}j_{h}\left(  2a\right)  .
\]
For the sum $\sum_{a=0}^{2^{i-1}-1}j_{h}\left(  2a\right)  $, there are two possibilities:

\begin{description}
\item[a] $\sum_{a=0}^{2^{i-1}-1}j_{h}\left(  2a\right)  =0$

\item[b] $\sum_{a=0}^{2^{i-1}-1}j_{h}\left(  2a\right)  =1.$
\end{description}

In case (a), $j_{g}\left(  0\right)  +$\ $j_{g}\left(  2^{i}\right)  =0$
implies $j_{g}\left(  0\right)  =$\ $j_{g}\left(  2^{i}\right)  $. Thus, the
equation $j_{g}\left(  x\right)  =$\ $j_{g}\left(  x+2^{i}\right)  $ holds for
$x=0$. By changing the initial value of $a$ between $0$ and $2^{i-1}-1$, we
verify that $j_{g}\left(  x\right)  =$\ $j_{g}\left(  x+2^{i}\right)  $ is
true when $x$ is even.

In case (b), the periodicity of the powers is not complete. Thus we continue:%
\[
j_{g}\left(  2a\right)  +j_{g}\left(  2a+2\right)  =j_{h}\left(  2a\right)
,\text{ \ }2^{i-1}\leq a<2^{i}-1.
\]
Summing up side by side, we obtain
\[
j_{g}\left(  2^{i}\right)  +2\text{\ }j_{g}\left(  2^{i}+2\right)
+...+j_{g}\left(  2^{i+1}\right)  =j_{h}\left(  2^{i}\right)  +...+j_{h}%
\left(  2^{i+1}-2\right)
\]
and
\[
j_{g}\left(  2^{i}\right)  +j_{g}\left(  2^{i+1}\right)  =\sum_{a=2^{i-1}%
}^{2^{i}-1}j_{h}\left(  2a\right)  =1.
\]
Finally, we put the two systems together:
\[
\left(  j_{g}\left(  0\right)  +j_{g}\left(  2^{i}\right)  \right)  +\left(
j_{g}\left(  2^{i}\right)  +j_{g}\left(  2^{i+1}\right)  \right)  =\sum
_{a=0}^{2^{i-1}-1}j_{h}\left(  2a\right)  +\sum_{a=2^{i-1}}^{2^{i}-1}%
j_{h}\left(  2a\right)  ,
\]
equivalently,%
\[
j_{g}\left(  0\right)  +2j_{g}\left(  2^{i}\right)  +j_{g}\left(
2^{i+1}\right)  =\sum_{a=0}^{2^{i}-1}j_{h}\left(  2a\right)  .
\]
Using the $2^{i-1}$-periodicity of the permutation representation of $h,$ we
have
\[
j_{g}\left(  0\right)  +j_{g}\left(  2^{i+1}\right)  =\sum_{a=0}^{2^{i}%
-1}j_{h}\left(  2a\right)  =0
\]
Therefore,
\[
j_{g}\left(  0\right)  =j_{g}\left(  2^{i+1}\right)  .
\]
Again, by changing the initial value of $a$, we can verify that $j_{g}\left(
x\right)  =j_{g}\left(  x+2^{i+1}\right)  $ for $x$ even.

Similar computations using%
\[
j_{g}\left(  2a+3\right)  +j_{g}\left(  2a+1\right)  =j_{h}\left(
2a+3\right)
\]
and induction yield the equality \ref{AC3} for odd $x$.
\end{proof}

\bigskip

Now we proceed to eliminate the assumption $\underset{x\in%
%TCIMACRO{\U{2115} }%
%BeginExpansion
\mathbb{N}
%EndExpansion
}{\prod}\left\langle \left(  2x\text{ \ }2x+1\right)  \right\rangle \leq
E_{k}\left(  H\right)  $ for all $k\in%
%TCIMACRO{\U{2115} }%
%BeginExpansion
\mathbb{N}
%EndExpansion
$ , from the statement of Proposition \ref{4.2}.

\begin{corollary}
\label{4.3}For every $k\in%
%TCIMACRO{\U{2115} }%
%BeginExpansion
\mathbb{N}
%EndExpansion
$, $\underset{x\in%
%TCIMACRO{\U{2115} }%
%BeginExpansion
\mathbb{N}
%EndExpansion
}{\prod}\left\langle \left(  2x\text{ \ }2x+1\right)  \right\rangle \leq
E_{k}\left(  H\right)  $; in particular, the assumption $\underset{x\in%
%TCIMACRO{\U{2115} }%
%BeginExpansion
\mathbb{N}
%EndExpansion
}{\prod}\left\langle \left(  2x\text{ \ }2x+1\right)  \right\rangle \leq
E_{k}\left(  H\right)  $ is superfluous in Proposition \ref{4.2}.
\end{corollary}

\begin{proof}
The statement is trivially true for $k=0.$ We assume that it holds for $k$.
For $k+1$, by definition,
\[
E_{k+1}(H)=\left\{  g\in E_{k}(H)\mid\left[  g,C_{E_{k}(H)}^{k+1}\left(
H\right)  \right]  \leq C_{E_{k}(H)}^{k}\left(  H\right)  \right\}  .
\]
By the inductive hypothesis $\underset{x\in%
%TCIMACRO{\U{2115} }%
%BeginExpansion
\mathbb{N}
%EndExpansion
}{\prod}\left\langle \left(  2x\text{ \ }2x+1\right)  \right\rangle \leq
E_{k}(H)$.$\ $Since $\underset{x\in%
%TCIMACRO{\U{2115} }%
%BeginExpansion
\mathbb{N}
%EndExpansion
}{\prod}\left\langle \left(  2x\text{ \ }2x+1\right)  \right\rangle $ is
abelian, and $C_{E_{k}(H)}^{k+1}\left(  H\right)  $ is a subgroup of
$\underset{x\in%
%TCIMACRO{\U{2115} }%
%BeginExpansion
\mathbb{N}
%EndExpansion
}{\prod}\left\langle \left(  2x\text{ \ }2x+1\right)  \right\rangle $ by
Proposition \ref{4.2}, we have
\[
\left[  g,C_{E_{k}(H)}^{k+1}\left(  H\right)  \right]  =1\leq C_{E_{k}(H)}%
^{k}\left(  H\right)
\]
for every $g\in\underset{x\in%
%TCIMACRO{\U{2115} }%
%BeginExpansion
\mathbb{N}
%EndExpansion
}{\prod}\left\langle \left(  2x\text{ \ }2x+1\right)  \right\rangle $. Thus,
$\underset{x\in%
%TCIMACRO{\U{2115} }%
%BeginExpansion
\mathbb{N}
%EndExpansion
}{\prod}\left\langle \left(  2x\text{ \ }2x+1\right)  \right\rangle $ is a
subgroup of $E_{k+1}(H)$.
\end{proof}

\bigskip

In the following two propositions, we will show that the iterated centralizers
$C_{E_{k}(H)}^{i}\left(  H\right)  $ form a specific ascending chain. In
Proposition\textbf{\ }\ref{4.5}\textbf{\ }we show that this ascendance is strict.

\bigskip

\begin{proposition}
For all $k\in%
%TCIMACRO{\U{2115} }%
%BeginExpansion
\mathbb{N}
%EndExpansion
$, $C_{E_{k+1}(H)}^{k+1}\left(  H\right)  =C_{E_{k}(H)}^{k+1}\left(  H\right)
$.
\end{proposition}

\begin{proof}
We will apply Lemma \ref{2.4}\textbf{\ }\textit{(iii)}, to the triple $H\leq
E_{k+1}\left(  H\right)  \leq E_{k}\left(  H\right)  $. By Lemma
\ref{2.6}\textbf{,} since $C_{E_{k}(H)}^{i}\left(  H\right)  =Z_{i}\left(
E_{k}(H)\right)  $ for $i,k\in%
%TCIMACRO{\U{2115} }%
%BeginExpansion
\mathbb{N}
%EndExpansion
$ and $i\leq k$, the hypotheses of Lemma \ref{2.4}\textbf{\ }is satisfied.
Thus
\[
C_{E_{k+1}(H)}^{k+1}\left(  H\right)  =C_{E_{k}(H)}^{k+1}\left(  H\right)
\cap E_{k+1}(H).
\]
The following inclusion resulting from Proposition \ref{4.2}\textbf{\ }and
Corollary \ref{4.3}\textbf{\ }%
\[
C_{E_{k}(H)}^{k+1}\left(  H\right)  \leq\underset{x\in%
%TCIMACRO{\U{2115} }%
%BeginExpansion
\mathbb{N}
%EndExpansion
}{\prod}\left\langle \left(  2x\text{ \ }2x+1\right)  \right\rangle \leq
E_{k+1}\left(  H\right)
\]
yields
\[
C_{E_{k+1}(H)}^{k+1}\left(  H\right)  =C_{E_{k}(H)}^{k+1}\left(  H\right)  .
\]

\end{proof}

\bigskip

\begin{proposition}
\label{4.5}For all $k\in%
%TCIMACRO{\U{2115} }%
%BeginExpansion
\mathbb{N}
%EndExpansion
$ and $i\leq k$, $C_{E_{k}(H)}^{i}\left(  H\right)  <C_{E_{k}(H)}^{i+1}\left(
H\right)  .$
\end{proposition}

\begin{proof}
We have shown in the proof of Proposition\textbf{\ }\ref{4.2}\textbf{\ }that
\[
C_{E_{k}\left(  H\right)  }\left(  H\right)  =\left\langle \underset{x\in%
%TCIMACRO{\U{2115} }%
%BeginExpansion
\mathbb{N}
%EndExpansion
}{\prod}\left(  2x\text{ \ }2x+1\right)  \right\rangle .
\]
Thus, each $C_{E_{k}(H)}^{i}\left(  H\right)  $ is nontrivial for every $k\in%
%TCIMACRO{\U{2115} }%
%BeginExpansion
\mathbb{N}
%EndExpansion
$ and $1\leq i\leq k+1.$ For every nontrivial $h\in C_{E_{k}(H)}^{i}\left(
H\right)  ,$ there is a first even (resp. odd) place $x_{0}$ such that
$j_{h}\left(  x_{0}\right)  =1$ and $j_{h}\left(  x\right)  =0$ for every even
$x$ (resp. odd) such that $x<x_{0}.$ We fix $h\in$ $C_{E_{k}(H)}^{i}\left(
H\right)  $ such that $x_{0}$ is maximal. Such an $h$ exists because
$C_{E_{k}(H)}^{i}\left(  H\right)  $ is finite \textbf{\ }by Proposition
\ref{4.2}. We will show that there exists $g\in$ $C_{E_{k}(H)}^{i+1}\left(
H\right)  \diagdown C_{E_{k}(H)}^{i}\left(  H\right)  $ such that $\left[
g,f\right]  =h.$ We will note $g=\underset{x\in%
%TCIMACRO{\U{2115} }%
%BeginExpansion
\mathbb{N}
%EndExpansion
}{\prod}\left(  2x\text{ \ }2x+1\right)  ^{j_{g}\left(  x\right)  }.$ We first
analyze the case $x_{0}$ is even.

By equation \ref{AC3}\textbf{\ }in\textbf{\ }Proposition \ref{4.2}%
\textbf{\ }applied to even places, we set up the following system of equations
for $g$ and $h:$%
\begin{align*}
\text{\ }j_{g}\left(  0\right)  +j_{g}\left(  2\right)   & =j_{h}\left(
0\right)  =0\\
\vdots\text{ \ \ \ \ \ \ \ \ \ }  & =\text{ \ \ \ \ \ \ \ }\vdots\\
j_{g}\left(  x_{0}-2\right)  +j_{g}\left(  x_{0}\right)   & =j_{h}\left(
x_{0}-2\right)  =0\\
j_{g}\left(  x_{0}\right)  +j_{g}\left(  x_{0}+2\right)   & =j_{h}\left(
x_{0}\right)  =1\\
\vdots\text{ \ \ \ \ \ \ \ \ \ }  & =\text{ \ \ \ \ \ \ \ }\vdots\\
j_{g}\left(  2^{i+1}-2\right)  +j_{g}\left(  2^{i+1}\right)   & =j_{h}\left(
2^{i+1}-2\right)  .
\end{align*}
This system of equations is then repeated since the places of $g$ are
$2^{i+1}$-periodic (Proposition \ref{4.2}).

We define $g$ coherently with this periodicity condition. We set $j_{g}\left(
0\right)  =0.$ Then $j_{g}\left(  2a\right)  =0$ for $2a\leq x_{0}$ and
$j_{g}\left(  x_{0}+2\right)  =1.$ The values of $j_{g}\left(  2a\right)  $
for $x+2\leq2a\leq2^{i+1}$are then computed inductively using the equation
system. We have to show that $j_{g}\left(  0\right)  =j_{g}\left(
2^{i+1}\right)  .$ Summing up the two sides of the equation system we obtain
\[
j_{g}\left(  0\right)  +j_{g}\left(  2^{i+1}\right)  =\sum_{x=0}^{2^{i+1}%
-2}j_{h}\left(  x\right)  =0.
\]
The second equality follows from the $2^{i}$-periodicity of $j_{h}.$

It remains to define $j_{g}\left(  2a+1\right)  $ for $a\in%
%TCIMACRO{\U{2115} }%
%BeginExpansion
\mathbb{N}
%EndExpansion
$ coherently with the $2^{i+1}$-periodicity of $j_{g}.$ From the equality
$j_{g}\left(  0\right)  +j_{g}\left(  1\right)  =j_{h}\left(  1\right)  $
(equation \ref{AC3}\textbf{), }one deduces the value of $j_{g}\left(
1\right)  $. The rest of the odd places of $j_{g}$ is inductively computed
using $j_{g}\left(  2a+3\right)  +j_{g}\left(  2a+1\right)  =j_{h}\left(
2a+3\right)  $ for $a\in%
%TCIMACRO{\U{2115} }%
%BeginExpansion
\mathbb{N}
%EndExpansion
$ (equation \ref{AC3}\textbf{).} The values of the $j_{g}\left(  2a+1\right)
$ are $2^{i+1}$-periodic because the right hand side of the equation $2^{i}$-periodic.

When $x_{0}$ is odd we use the equation
\[
j_{g}\left(  0\right)  +j_{g}\left(  1\right)  =j_{h}\left(  1\right)
\]
in \ref{AC3} to set $j_{g}\left(  1\right)  =0$ and $j_{g}\left(  0\right)
=1$. After this initial step, the rest of the the values of $j_{g}$ are
determined coherently using the two recursive equalities of \ref{AC3}, in a
similar way to the case when $x_{0}$ is even.

Finally, we remark that $g\in$ $C_{E_{k}(H)}^{i+1}\left(  H\right)  \diagdown
C_{E_{k}(H)}^{i}\left(  H\right)  $. Indeed, by construction $\left[
g,f\right]  \in C_{E_{k}(H)}^{i}\left(  H\right)  $ and trivially $\left[
g,K\right]  =1$, hence $g\in C_{E_{k}(H)}^{i+1}\left(  H\right)  $. Moreover,
by the maximal choice of $x_{0},$ $g$ is not in $C_{E_{k}(H)}^{i}\left(
H\right)  .$
\end{proof}

\bigskip

\bigskip The two preceding propositions yield the following corollary:

\begin{corollary}
\label{4.6}The following relations hold between the iterated centralizers such
that $j<j^{^{\prime}},$ $k\leq k^{^{\prime}},$ $j\leq k,$ $j^{^{\prime}}\leq
k^{^{\prime}}$ for all natural numbers:

\begin{enumerate}
\item[(i)] $C_{E_{k}\left(  H\right)  }^{j}\left(  H\right)  =C_{E_{k^{\prime
}}\left(  H\right)  }^{j}\left(  H\right)  $

\item[(ii)] $C_{E_{k}\left(  H\right)  }^{j}\left(  H\right)  <C_{E_{k^{\prime
}}\left(  H\right)  }^{j^{\prime}}\left(  H\right)  .$
\end{enumerate}

\.{I}n particular, the chain $\left(  C_{E_{k}\left(  H\right)  }^{j}\left(
H\right)  \right)  _{\left(  k,j\right)  }$ is strictly increasing, the
indices $\left(  k,j\right)  $ being ordered lexicographically.
\end{corollary}

\begin{proof}
We know $C_{E_{k}\left(  H\right)  }^{j}\left(  H\right)  \leq C_{E_{k}\left(
H\right)  }^{k}\left(  H\right)  $ for $j\leq k$ by properties of iterated
centralizers. By Lemma \ref{2.6}\textbf{,} $C_{E_{k}\left(  H\right)  }%
^{k}\left(  H\right)  =Z_{k}\left(  E_{k}\left(  H\right)  \right)  $. Since
the inclusion $Z_{k}\left(  E_{k}\left(  H\right)  \right)  \leq Z_{k^{\prime
}}\left(  E_{k^{\prime}}\left(  H\right)  \right)  $ holds for $k\leq
k^{\prime}$ by Corollary \ref{2.5}, we get
\begin{equation}
C_{E_{k}\left(  H\right)  }^{j}\left(  H\right)  \leq C_{E_{k}\left(
H\right)  }^{k}\left(  H\right)  =Z_{k}\left(  E_{k}\left(  H\right)  \right)
\leq Z_{k^{\prime}}\left(  E_{k^{\prime}}\left(  H\right)  \right)  \leq
E_{k^{\prime}}\left(  H\right)  .\label{AC4}%
\end{equation}

Since $C_{E_{k}\left(  H\right)  }^{j}\left(  H\right)  =Z_{j}\left(
E_{k}\left(  H\right)  \right)  $ for $j\leq k,$ we can apply Lemma \ref{2.4}
to $H\leq E_{k^{\prime}}\left(  H\right)  \leq E_{k}\left(  H\right)  $. Thus
$C_{E_{k^{\prime}}\left(  H\right)  }^{j}\left(  H\right)  =C_{E_{k}\left(
H\right)  }^{j}\left(  H\right)  \cap E_{k^{\prime}}\left(  H\right)  $ for
all $j\leq k.$

Since we have already shown $C_{E_{k}\left(  H\right)  }^{j}\left(  H\right)
\leq E_{k^{\prime}}\left(  H\right)  ,$ we get $C_{E_{k}\left(  H\right)
}^{j}\left(  H\right)  =C_{E_{k^{\prime}}\left(  H\right)  }^{j}\left(
H\right)  $; \textit{(i)} follows. Conclusion \textit{(ii)} is a consequence
of Proposition\textbf{\ }\ref{4.5}\textbf{\ }and conclusion \textit{(i)}.
\end{proof}

\bigskip

\bigskip We start the final part of our proof that the $E_{k}\left(  H\right)
$ form an infinite decreasing chain. We first prove a technical lemma.

\begin{lemma}
\label{4.8}Let $k\in%
%TCIMACRO{\U{2115} }%
%BeginExpansion
\mathbb{N}
%EndExpansion
$. Let $l\in%
%TCIMACRO{\U{2115} }%
%BeginExpansion
\mathbb{N}
%EndExpansion
$ be minimal such that for every $g\in$ $C_{E_{k}}^{k+1}\left(  H\right)  $
and for every $x\in%
%TCIMACRO{\U{2115} }%
%BeginExpansion
\mathbb{N}
%EndExpansion
$, $j_{g}\left(  x\right)  =j_{g}\left(  x+2^{l}\right)  $. Then
\[
G_{x,l}\leq E_{k}%
\]
where
\[
G_{x,l}=\left\langle \left(  2\left(  x+2^{l}i\right)  \text{ }2\left(
x+2^{l}i^{^{\prime}}\right)  \right)  \left(  \left(  2\left(  x+2^{l}%
i\right)  +1\right)  \left(  2\left(  x+2^{l}i^{^{\prime}}\right)  +1\right)
\right)  \mid i,i^{^{\prime}}\in%
%TCIMACRO{\U{2115} }%
%BeginExpansion
\mathbb{N}
%EndExpansion
\right\rangle
\]
and $0\leq x<2^{l}$.
\end{lemma}

\begin{proof}
We proceed by induction on $k$. For $k=0$, the conlusion is trivial. We
inductively assume that the claim holds for $k.$ In particular, $G_{x,l}\leq
E_{k}.$ It is known from Proposition\textbf{\ }\ref{4.2} that for every $h\in
C_{E_{k+1}\left(  H\right)  }^{k+2}\left(  H\right)  $ there exist $l^{\prime
}\leq k+2$ such that for every $0\leq x<2^{l^{\prime}},$ $j_{h}\left(
x\right)  =j_{h}\left(  x+2^{l^{\prime}}\right)  .$ We want to show the
inclusion $G_{x,l^{^{\prime}}}\leq E_{k+1}$. Let $g\in G_{x,l^{^{\prime}}}$ .
In order to verify $g\in E_{k+1}$, there are two conditions to check:

\begin{description}
\item[a] $g\in E_{k}$,

\item[b] $\left[  g,C_{E_{k}}^{k+1}\left(  H\right)  \right]  \leq
C_{E_{k}(H)}^{k}\left(  H\right)  $.

\item[a] By induction, $G_{x,l}\leq E_{k}$ for $0\leq x<2^{l}$. Since by
Corollary \ref{4.6}\textbf{\ }$C_{E_{k}(H)}^{k+1}\left(  H\right)  \leq
C_{E_{k+1}(H)}^{k+2}\left(  H\right)  ,$ we have $l\leq l^{\prime}.$ Therefore
$2^{l}$ divides $2^{l^{\prime}}$ and we get $G_{x,l^{^{\prime}}}\leq
G_{x,l}\leq E_{k}$ for $0\leq x<2^{l^{\prime}}.$

\item[b] By definition, $g\in G_{x,l^{^{\prime}}}$ for a certain $x$ such
that\ $0\leq x<2^{l^{\prime}}$. Since $2^{l}$ divides $2^{l^{\prime}}$ ,
$\left[  g,C_{E_{k}(H)}^{k+1}\left(  H\right)  \right]  =1$. Hence $\left[
g,C_{E_{k}(H)}^{k+1}\left(  H\right)  \right]  \subseteq C_{E_{k}(H)}%
^{k}\left(  H\right)  $.
\end{description}
\end{proof}

\bigskip

\begin{teo}
For every $k\in%
%TCIMACRO{\U{2115} }%
%BeginExpansion
\mathbb{N}
%EndExpansion
$, there is at least one $k^{^{\prime}}>k$ such that $E_{k^{\prime}+1}%
<E_{k+1}$.
\end{teo}

\begin{proof}
We start by fixing $k\in%
%TCIMACRO{\U{2115} }%
%BeginExpansion
\mathbb{N}
%EndExpansion
.$ Let $x,l$ and $G_{x,l}\leq E_{k+1}$ be as in Lemma \ref{4.8}. By
Proposition \ref{4.2}, for every $h\in C_{E_{k}}^{k+1}\left(  H\right)  $,
$j_{h}\left(  x\right)  =j_{h}\left(  x+2^{l}\right)  $. Since $C_{E_{k}}%
^{i}\left(  H\right)  $ is finite and by Corollary \ref{4.6}\textbf{\ }the
chain of such iterated centralizers is strictly increasing, there exist
$k^{\prime}>k$ , $h\in C_{E_{k^{\prime}}(H)}^{k^{\prime}+1}\left(  H\right)  $
and $x_{0}$ , $0\leq x_{0}<2^{l},$ such that $j_{h}\left(  x_{0}\right)  \neq
j_{h}\left(  x_{0}+2^{l}\right)  $. We define $g$ to be $(2x_{0}$
$2(x_{0}+2^{l}))(2x_{0}+1$ $2(x_{0}+2^{l})+1)$. By Lemma \ref{4.8} $g\in
E_{k}$. In fact, $g\in E_{k+1}\left(  H\right)  $ since $\left[
g,C_{E_{k}(H)}^{k+1}\left(  H\right)  \right]  =1\subseteq C_{E_{k}(H)}%
^{k}\left(  H\right)  $. But $g\notin E_{k^{\prime}+1}$. Indeed, one computes
\begin{align*}
\left[  g,h\right]   & =h^{g}h\\
& =\left(  2x_{0}\text{ }2x_{0}+1\right)  \left(  2\left(  x_{0}+2^{l}\right)
\text{\ }2\left(  x_{0}+2^{l}\right)  +1\right)  \left(  \underset{i\neq
x_{0},\text{\ }x_{0}+2^{l}}{\prod}\left(  2i\text{ }2i+1\right)
^{j_{h}\left(  i\right)  }\right)  ^{2}\\
& =\left(  2x_{0}\text{ \ }2x_{0}+1\right)  \text{ }\left(  2\left(
x_{0}+2^{l}\right)  \text{ \ }2\left(  x_{0}+2^{l}\right)  +1\right)  .
\end{align*}
for $i\in%
%TCIMACRO{\U{2115} }%
%BeginExpansion
\mathbb{N}
%EndExpansion
.$ Since $\left[  g,h\right]  \neq1$ and is of finite support, $\left[
g,h\right]  \notin C_{E_{k^{\prime}}(H)}^{k^{\prime}}\left(  H\right)  $.
Hence $g\notin E_{k^{\prime}+1}(H)$.
\end{proof}

\section{\bigskip Concluding Remarks}

Our final theorem shows that the subgroup $H\leq Sym\left(
%TCIMACRO{\U{2115} }%
%BeginExpansion
\mathbb{N}
%EndExpansion
\right)  $ is a counterexample to the stabilization problem. Indeed, in the
proof of the main theorem $k$ was arbitrary, and as a result, the descending
chain $\left(  E_{k}\left(  H\right)  \right)  _{k}$ is infinite. We recall
that $Sym\left(
%TCIMACRO{\U{2115} }%
%BeginExpansion
\mathbb{N}
%EndExpansion
\right)  $ is not an $\mathfrak{M}_{C}$-group. Whether the stabilization
problem has an affirmative answer in groups satisfiying the chain condition on
centralizers remains open.

\section{\bigskip Acknowledgements}

The author is indebted to her Ph.D. supervisors Erdal Karaduman and Tuna Alt\i
nel for their continuous support and infinite patience. The author would like
to thank to the referee for valuable comments and suggestions that improved
the presentation of the paper.

%\bibliography{acompat,articles}

\end{document}